\documentclass[13pt,reqno]{amsart}
\UseRawInputEncoding

\usepackage{amsmath,amssymb,amsthm}
\usepackage{bm}
\usepackage{graphicx}
\usepackage{color}
\usepackage{float}
\usepackage{url}
\usepackage{stmaryrd}
\usepackage{mathrsfs}
\usepackage{caption}
\usepackage{url}
\newcommand{\bel}[1]{\begin{equation*}\label{#1}}
	
	\newcommand{\be}{\begin{equation}}

		\newcommand{\ba}{\begin{eqnarray}}
			\newcommand{\ea}{\end{eqnarray}}

		\newcommand{\qe}{\end{equation}}
	\newcommand{\R}{{\mathbb R}}

	\newcommand{\eg}{\begin{example}}
		\newcommand{\egd}{\end{example}}
	\newcommand{\tm}{\begin{thm}}
		\newcommand{\tmd}{\end{thm}}
	\newcommand{\co}{\begin{coro}}
		\newcommand{\cod}{\end{coro}}
	\newcommand{\enu}{\begin{enumerate}}
		\newcommand{\enud}{\end{enumerate}}
	\newcommand{\rmk}{\begin{rem}}
		\newcommand{\rmkd}{\end{rem}}
	
	\theoremstyle{theorem}
	\newtheorem{thm}{Theorem}[section]
	\newtheorem{prop}[thm]{Proposition}
	\theoremstyle{example}
	\newtheorem{example}[thm]{Example}
	\newtheorem{coro}[thm]{Corollary}
	\theoremstyle{lemma}
	\newtheorem{lemma}[thm]{Lemma}
	\theoremstyle{definition}
	\newtheorem{defi}[thm]{Definition}
	\theoremstyle{proof}
	
	\theoremstyle{remark}
	\newtheorem{rem}[thm]{Remark}
	\theoremstyle{remark}

	\UseRawInputEncoding

\begin{document}
		
		\title[Local existence of strong solutions to the capillary wave kinetic equation]{Local existence of strong solutions to the capillary wave kinetic equation}
		\author{Jiajun Wang}
		\address{Jiajun Wang: Courant Institute of Mathematical Sciences,
			New York University, New York, NY}
		\email{jw9409@nyu.edu}
		
		\begin{abstract}
			In this paper, we follow Pan-Wu's idea \cite{pan2026local} to prove local existence of the capillary wave kinetic equation. To be more specific, we will linearize the nonlinear operator and decompose it into dissipative part and bounded part. The main difficulty comes from the fact that the collision kernel has less symmetry compared with the one in gravity water wave. To overcome it, we are required to choose and pair each term delicately.
		\end{abstract}
		
		\maketitle
		\vspace{-0.7cm}
		\numberwithin{equation}{section}
		
		\section{Introduction}
	Wave turbulence theory (WTT) is a statistical theory describing weakly nonlinear dispersive waves and their interactions. In the weakly nonlinear regime, nonlinear effects induce stochasticity in the wave phases together with a slow modulation of the amplitudes, allowing for a statistical description of large ensembles of interacting waves. Over the past several decades, WTT has undergone extensive development.
	
	The wave kinetic equation (WKE) forms the foundation of WTT, governing the long-time evolution of the wave action spectrum in momentum space. Recently, the rigorous derivation and justification of WKE have attracted unprecedented attention from the mathematical community. A series of significant advances have been made by Deng and Hani in \cite{deng2021derivation,deng2023full,deng2023long}. For a more detailed introduction to these developments, we refer the reader to \cite{pan2026local} and the references therein.
		
	In general, WKE is typically classified into three-wave and four-wave types. Historically, the first kinetic equation for weak turbulence was a three-wave kinetic equation, derived in the context of phonon interactions in anharmonic crystal lattices \cite{peierls1929kinetischen,peierls1955quantum}. The kinetic equation for capillary waves also takes the form of a three-wave kinetic equation.

	In particular, the capillary wave kinetic equation (\(\omega_k=|k|^{3/2}\)) is given by

		\begin{equation}\label{Capillary}
			\left\{
			\begin{aligned}
				\frac{\partial n_k}{\partial t}
				&=S(n_k), \qquad n_k(0)=n_0(k)\ge0,\\
				S(n_k)
				&=
				\iint_{\R^{2d}}
				\left(
				R_{k k_1 k_2}
				-
				R_{k_1 k k_2}
				-
				R_{k_2 k k_1}
				\right)
				\,dk_1\,dk_2,
			\end{aligned}
			\right.
		\end{equation}
		where the collision integrand is defined as 
		
		\begin{equation*}
			R_{k k_1 k_2}
			=
			4\pi
			|V_{k k_1 k_2}|^2
			\delta(k-k_1-k_2)
			\delta(\omega_k-\omega_{k_1}-\omega_{k_2})
			\left(
			n_{k_1}n_{k_2}
			-
			n_kn_{k_1}
			-
			n_kn_{k_2}
			\right),
		\end{equation*}
		
		\begin{equation*}
			V_{k k_1 k_2}
			=
			\frac{1}{8\pi\sqrt{2}}
			(\omega_k\omega_{k_1}\omega_{k_2})^{1/2}
			\left[
			\frac{L_{k_1,k_2}}
			{|k_1|^{1/2}|k_2|^{1/2}|k|}
			-
			\frac{L_{k,-k_1}}
			{|k|^{1/2}|k_1|^{1/2}|k_2|}
			-
			\frac{L_{k,-k_2}}
			{|k|^{1/2}|k_2|^{1/2}|k_1|}
			\right],
		\end{equation*}
		
		\begin{equation*}
			L_{k_1,k_2}
			=
			k_1\cdot k_2
			+
			|k_1||k_2|.
		\end{equation*}

		For four-wave kinetic equations, surface gravity waves in fluids of finite depth provide a classical example. Hasselmann established the first wave kinetic equation for gravity waves in \cite{hasselmann1962non}, which has been widely used in modern wave forecasting. Zakharov also developed a statistical theory for gravity waves \cite{zakharov1999statistical}. Other four-wave models, such as those describing Langmuir waves \cite{zakharov1972collapse} and quantum fluids \cite{kolmakov2014wave}, also play important roles in WTT.
		
	\vspace{5pt}
	Back to capillary waves. The derivation of the capillary wave kinetic equation can be traced back to \cite{zakharov1965weak,zakharov1967weak,zakharov2012kolmogorov}. Pushkarev and Zakharov \cite{pushkarev2000turbulence} performed direct numerical simulations of capillary wave turbulence, confirming the predicted Kolmogorov spectrum \(I_k \propto k^{-19/4}\) for the surface elevation. Pan \cite{pan2017understanding} subsequently pointed out several errors in \cite{pushkarev2000turbulence} concerning the analytical evaluation of the Kolmogorov constant. In addition, \cite{pan2014direct} presented a direct numerical study of capillary waves. For further discussion and related topics, we refer the reader to \cite{nazarenko2011wave}.
	
	However, the theoretical analysis of the capillary wave kinetic equation remains rather limited, and many fundamental problems are still open due to its intricate structure. To the best of the author's knowledge, \cite{nguyen2018kinetic} is the only work that studies the capillary wave kinetic equation from a purely mathematical perspective. In that paper, the authors first developed many new techniques inspired by the progress made on the quantum Boltzmann equation \cite{escobedo2015convergence, alonso2016cauchy, nguyen2019uniform}.  Moreover, they applied those techniques to establish global existence and uniqueness in the radial setting.
		
		\vspace{7pt}
		Next, we briefly introduce the function spaces that will be used throughout this paper. These spaces are standard in WTT.
		
		The weighted Lebesgue space $L_s^p(\R^d)$, $s\ge 0$, is defined as 
		\begin{equation*}
			\left\{f:\R^d\to \R\;\bigg|\; \int_{\R^d}\langle k\rangle^{ps}\cdot |f(k)|^{p}dk<\infty\right\},
		\end{equation*}
		where $\langle k\rangle:=\left(1+|k|^2\right)^\frac{1}{2}$. And we equip $L_s^p(\R^d)$ with the norm
		\begin{equation*}
			\|f\|_{L_s^p(\R^d)}:=\left(\int_{\R^d}\langle k\rangle^{ps}\cdot |f(k)|^{p}dk\right)^{1/p}.
		\end{equation*}
		
		We also use the following notational conventions: $\|\cdot\|_p\equiv \|\cdot\|_{L^p(\R^d)}$, $\|\cdot\|_{p\to p}\equiv \|\cdot\|_{L^{p}(\R^d)\to L^{p}(\R^d)}$, $\|\cdot\|_{p,s}\equiv \|\cdot\|_{L_s^p(\R^d)}$ and $\|\cdot\|_{p,s\to p,s}\equiv \|\cdot\|_{L_s^p(\R^d)\to L_s^p(\R^d)}$.

\vspace{7pt}
The main result of this paper is the local existence of a strong solution to (\ref{Capillary}).

\begin{thm}\label{main}
	For $d\ge 2$, $n_0\in L_{4d+12}^{\infty}(\R^d)$, there exists $T=T\left(\|n_0\|_{\infty,4d+12}\right)>0$, such that the capillary wave kinetic equation (\ref{Capillary}) has a strong solution 
	\begin{equation*}
		n_k=n(t,k)\in C\left([0,T]; L_{2(d+1)}^{\infty}(\R^d)\right)\subseteq C\left([0,T]; L^{1}(\R^d)\right).
	\end{equation*}
	Moreover, we have the uniform bound
	\begin{equation*}
		\sup_{t\in[0,T]}\|n(t,k)\|_{\infty,4d+12}\le 2\|n_0\|_{\infty,4d+12}.
	\end{equation*}
\end{thm}
\begin{rem}
Note that \(n_k\) represents the wave density in kinetic theory and is therefore non-negative. Moreover, the system \eqref{Capillary} preserves non-negativity throughout its evolution (see, e.g., \cite{nazarenko2011wave}). We will also provide an informal argument for this property in Remark~\ref{nonnegative}. Accordingly, throughout this paper, we restrict our attention to non-negative solutions, a condition that is essential for establishing dissipativity.
\end{rem}
\begin{rem}
   The main contribution of the present work is the removal of the isotropy assumption imposed in \cite{nguyen2018kinetic}, thereby establishing the local well-posedness of the capillary wave kinetic equation in the general, non-isotropic setting. We also emphasize that the weight exponents \(4d+12\) and \(2(d+1)\) in Theorem~\ref{main} are chosen sufficiently large to close the arguments and are not intended to be optimal.
\end{rem}

\vspace{7pt}
Next, we define the linearized operator $Q_g(t)$ acting on $h$ as follows:
\begin{equation*}
	Q_g(t)h:=\iint_{\R^{2d}} 4\pi
	|V_{k k_1 k_2}|^2
	\delta(k-k_1-k_2)
	\delta(\omega-\omega_{1}-\omega_{2}) \left[2g_1h_2 \chi_{2\ge 1}-g_1h-g_2h\right]dk_1dk_2
\end{equation*}

\begin{equation*}
	-\iint_{\R^{2d}} 4\pi
	|V_{k_1 k k_2}|^2
	\delta(k_1-k-k_2)
	\delta(\omega_{1}-\omega-\omega_{2}) \left[g_2h-g_1h-g_2h_1\right]dk_1dk_2
\end{equation*}

\begin{equation}\label{linear}
	-\iint_{\R^{2d}} 4\pi
	|V_{k_2 k k_1}|^2
	\delta(k_2-k-k_1)
	\delta(\omega_{2}-\omega-\omega_{1}) \left[g_1h-g_2h-g_1h_2\right]dk_1dk_2.
\end{equation}
Here, we follow the shorthand $\chi_{2\ge 1}:=\chi_{|k_2|\ge |k_1|}$ and 
\begin{equation*}
	\omega=\omega_k=|k|^{3/2}, \quad \omega_1=\omega_{k_1}=|k_1|^{3/2}, \quad \omega_2=\omega_{k_2}=|k_2|^{3/2},
\end{equation*}
\begin{equation*}
	g=g(k),\; g_i=g(k_i), \quad h=h(k),\; h_i=h(k_i), \quad \forall i=1,2.
\end{equation*}
One can also check that $Q_{n_k}(t)n_k=S(n_k)$, so $Q_g(t)$ is truly a linearization for $S$.

\vspace{7pt}
Then we can decompose the linearized operator $Q_g(t)$ into a dissipative part and a bounded part:

\begin{equation*}
	Q_{g,D}(t)h:=\iint_{\R^{2d}} 4\pi
	|V_{k k_1 k_2}|^2
	\delta(k-k_1-k_2)
	\delta(\omega-\omega_{1}-\omega_{2}) \left[2g_1h_2 \chi_{2\ge 1}-g_1h-g_2h\right]dk_1dk_2
\end{equation*}
\begin{equation*}
	-\iint_{\R^{2d}} 4\pi
	|V_{k_1 k k_2}|^2
	\delta(k_1-k-k_2)
	\delta(\omega_{1}-\omega-\omega_{2}) \left[g_2h-g_2h_1\right]dk_1dk_2
\end{equation*}

\begin{equation}\label{dissi}
	-\iint_{\R^{2d}} 4\pi
	|V_{k_2 k k_1}|^2
	\delta(k_2-k-k_1)
	\delta(\omega_{2}-\omega-\omega_{1}) \left[g_1h-g_1h_2\right]dk_1dk_2.
\end{equation}
\vspace{7pt}
\begin{equation*}
	Q_{g,b}(t)h:=\iint_{\R^{2d}} 4\pi
	|V_{k_1 k k_2}|^2
	\delta(k_1-k-k_2)
	\delta(\omega_{1}-\omega-\omega_{2}) g_1h \; dk_1dk_2
\end{equation*}
\begin{equation}\label{bounded}
	+\iint_{\R^{2d}} 4\pi
	|V_{k_2 k k_1}|^2
	\delta(k_2-k-k_1)
	\delta(\omega_{2}-\omega-\omega_{1}) g_2h \; dk_1dk_2.
\end{equation}
Note that $V_{kk_1k_2}=V_{kk_2k_1}$; one can directly verify that 
\begin{equation*}
	Q_g(t)=Q_{g,D}(t)+Q_{g,b}(t).
\end{equation*}

\vspace{10pt}
To establish the local existence, we utilize the following iteration scheme (conventionally, we set $n_1(t,k)\equiv n_0(k)$):

\begin{equation}\label{iter}
	\begin{cases}
		\dfrac{\partial n_{\ell+1}}{\partial t}=Q_{\ell}(t)n_{\ell+1}, \quad Q_{\ell}(t):=Q_{n_{\ell}}(t), \\[4pt]
		n_{\ell+1}(0,k)=n_0(k), \qquad (t,k)\in \R^{+} \times \R^d.
	\end{cases}
\end{equation}
\begin{rem}\label{nonnegative}
	This iterative scheme also preserves non-negativity due to the intrinsic kinetic structure of the linearized operator. This property can be understood from an ODE perspective.
	
	Indeed, the linearized operator \(Q_g\) can be decomposed into the sum of a positive integral operator and a multiplication operator:
	\begin{equation*}
		(Q_g(t)h)(k):=\iint_{\R^{2d}}K_1(k,k_1,k_2)h(k_1)+K_2(k,k_1,k_2)h(k_2)dk_1dk_2+\nu(k)h(k)
	\end{equation*}
	\begin{equation*}
		=:(\mathcal{I}h)(k)+\nu(k)h(k),
	\end{equation*}
	where the kernels $K_1(k,k_1,k_2)$ and $K_2(k,k_1,k_2)$ are given by
	\begin{equation*}
		K_1(k,k_1,k_2):=4\pi
		|V_{k_1 k k_2}|^2
		\delta(k_1-k-k_2)
		\delta(\omega_1-\omega-\omega_{2})g_2,
	\end{equation*}
	\begin{equation*}
		K_2(k,k_1,k_2):=8\pi
		|V_{k k_1 k_2}|^2
		\delta(k-k_1-k_2)
		\delta(\omega-\omega_{1}-\omega_{2})g_1+4\pi
		|V_{k_2 k k_1}|^2
		\delta(k_2-k-k_1)
		\delta(\omega_2-\omega-\omega_{1})g_1.
	\end{equation*}
	Note that when $g$ is non-negative, the kernels are also non-negative. 
	
	Without loss of generality, one can further assume that $g$ and $n_0$ are strictly positive, since we can apply continuity and perturbation.
	
	let \(t_0>0\) be the first time at which \(h\) attains the value \(0\). Then there exists \(k_0\in\mathbb{R}^d\) such that
	\[
	h(t_0,k_0)=0,
	\qquad
	h(t_0,k)\ge0,\quad \forall\,k\in\mathbb{R}^d.
	\]
	Since the kernels \(K_1\) and \(K_2\) are non-negative and $g:=n_{\ell}$ is strictly positive (by induction),
	\[
	(\mathcal{I}h)(t_0,k_0)>0,
	\]
	while
	\[
	(\nu h)(t_0,k_0)=\nu(k_0)h(t_0,k_0)=0.
	\]
	Hence,
	\[
	\partial_t h(t_0,k_0)=(Q_\ell(t_0)h)(k_0)
	=(\mathcal{I}h)(t_0,k_0)>0,
	\]
	contradicting the definition of \(t_0\). Therefore,
	\[
	h(t,k)\ge0,\quad \forall t\ge0.
	\]
\end{rem}

\vspace{7pt}
For any weight exponent \(a\ge0\), the weighted function \(\langle k\rangle^{a}n_{\ell+1}\) satisfies the following evolution equation:

\begin{equation}\label{weighted equation}
	\begin{cases}
		\dfrac{\partial }{\partial t}(\langle k\rangle^{a} n_{\ell+1})=[Q_{\ell}(t)+\mathcal{R}_{\ell}(t)](\langle k\rangle^{a} n_{\ell+1}),  \\[4pt]
		\langle k\rangle^{a} n_{\ell+1}(0,k)=\langle k\rangle^{a}n_0(k), \qquad (t,k)\in \R^{+} \times \R^d,
	\end{cases}
\end{equation}
where the extra operator $\mathcal{R}_{\ell}(t)=\mathcal{R}_{n_\ell}(t)$ is given by
\begin{equation*}
	\mathcal{R}_{g}(t) \varphi:= \iint_{\R^{2d}} 8\pi
	|V_{k k_1 k_2}|^2
	\delta(k-k_1-k_2)
	\delta(\omega-\omega_{1}-\omega_{2}) g_1 \dfrac{\langle k\rangle^{a}-\langle k_2\rangle^{a}}{\langle k_2\rangle^{a}} \varphi_2 \chi_{2\ge 1} \; dk_1dk_2
\end{equation*}
\begin{equation*}
	+\iint_{\R^{2d}} 4\pi
	|V_{k_1 k k_2}|^2
	\delta(k_1-k-k_2)
	\delta(\omega_1-\omega-\omega_{2}) g_2 \dfrac{\langle k\rangle^{a}-\langle k_1\rangle^{a}}{\langle k_1\rangle^{a}} \varphi_1 \; dk_1dk_2
\end{equation*}
\begin{equation}\label{extra}
	+\iint_{\R^{2d}} 4\pi
	|V_{k_2 k k_1}|^2
	\delta(k_2-k-k_1)
	\delta(\omega_2-\omega-\omega_{1}) g_1 \dfrac{\langle k\rangle^{a}-\langle k_2\rangle^{a}}{\langle k_2\rangle^{a}} \varphi_2 \; dk_1dk_2.
\end{equation}

\vspace{10pt}
\noindent
\textbf{Organization.}

This paper is organized as follows. In Section~2, we establish a sharp estimate for the collision kernel \(V_{kk_1k_2}\), which plays a central role in the subsequent sections. In Section~3, we prove the dissipativity of \(Q_{g,D}\) and the boundedness of \(Q_{g,b}\) and \(\mathcal{R}_g\). In Section~4, we combine the results established in Section~3 to derive a key energy estimate. In Section~5, we present the proof of Theorem~\ref{main} in detail. Finally, Appendix~A is devoted to the proofs of several auxiliary estimates used throughout the paper, while Appendix~B rigorously justifies the construction of the propagator generated by the dissipative operator.

\vspace{10pt}
\noindent
\textbf{Notation.}
\begin{itemize}
	\item By $u\in C([0,T]; B) \; ( \mbox{or} \; L^{p}([0,T];B))$ for a Banach space $B,$ we mean $u$ is a continuous $(\mbox{or} \; L^{p})$ map from $[0,T]$ to $B;$ see \cite[page 301]{evans2022partial}.
	\item By $A\lesssim B$ (resp. $A\sim B$), we mean there is a positive constant $C$, such that $A\le CB$ (resp. $C^{-1}B\le A \le C B$). If the constant $C$ depends on $p,$ then we write $A\lesssim_{p}B$ (resp. $A\sim_{p} B$).
	\item By $A\ll B$, we mean $\frac{A}{B}$ is sufficiently smaller than $1$.
	\item By $\mathcal{S}(\R^d)$, we mean the Schwarz space on $\R^d$.
\end{itemize}

\vspace{10pt}
\section{Sharp estimate for the collision kernel}
To control the growth of the strong solution, we need an estimate for the collision kernel $|V_{kk_1k_2}|^2$, where  
\begin{equation*}
	V_{k k_1 k_2}
	=
	\frac{1}{8\pi\sqrt{2}}
	(\omega_k\omega_{k_1}\omega_{k_2})^{1/2}
	\left[
	\frac{L_{k_1,k_2}}
	{|k_1|^{1/2}|k_2|^{1/2}|k|}
	-
	\frac{L_{k,-k_1}}
	{|k|^{1/2}|k_1|^{1/2}|k_2|}
	-
	\frac{L_{k,-k_2}}
	{|k|^{1/2}|k_2|^{1/2}|k_1|}
	\right].
\end{equation*}

\vspace{7pt}
By using Taylor expansion, we can obtain the following sharp estimate:
\begin{prop}\label{collision kernel}
	Suppose
	\[
	k=k_1+k_2,\qquad
	|k|^{3/2}=|k_1|^{3/2}+|k_2|^{3/2},
	\]
	
	Then we have
	\[
	\begin{aligned}
		&
		\Bigg|\frac{L_{k_1,k_2}}
		{(|k_1||k_2|)^{1/2}|k|}
		-
		\frac{L_{k,-k_1}}
		{(|k||k_1|)^{1/2}|k_2|}
		-
		\frac{L_{k,-k_2}}
		{(|k||k_2|)^{1/2}|k_1|}\Bigg|
		\lesssim 
		\frac{\min\{|k_1|,|k_2|\}}
		{\max\{|k_1|,|k_2|\}}.
	\end{aligned}
	\]
\end{prop}

\begin{proof}
	If $|k_1|,|k_2|\sim |k|$, then the above estimate holds trivially. Without loss of generality, we can assume
	\[
	|k_1|\ll |k_2|.
	\]
	Let
	\[
	m:=|k_1|,\qquad
	M:=|k_2|,\qquad
	r:=\frac{m}{M},
	\]
	so that \(r\ll1\). By the resonance relation,
	\[
	|k|
	=
	M(1+r^{3/2})^{2/3}.
	\]
	
	Using
	\[
	k_1\cdot k_2
	=
	\frac{|k|^2-m^2-M^2}{2},
	\]
	the left-hand side can be written as
	\[
	B(r)=T_1(r)-T_2(r)-T_3(r),
	\]
	where
	\[
	\begin{aligned}
		T_1&=
		\frac{r+\frac12(s^2-r^2-1)}
		{r^{1/2}s},
		\\
		T_2&=
		\frac{sr-\frac12(r^2+s^2-1)}
		{(rs)^{1/2}},
		\\
		T_3&=
		\frac{s-\frac12(1+s^2-r^2)}
		{s^{1/2}r},
	\end{aligned}
	\qquad
	s=(1+r^{3/2})^{2/3}.
	\]
	
	Expanding
	\[
	s
	=
	1+\frac23r^{3/2}
	-\frac19r^3
	+O(r^{9/2}),
	\]
	one obtains
	\[
	T_1
	=
	r^{1/2}
	+\frac23r
	-\frac12r^{3/2}
	+O(r^{5/2}),
	\]
	\[
	T_2
	=
	r^{1/2}
	-\frac23r
	-\frac12r^{3/2}
	+O(r^{5/2}),
	\]
	and
	\[
	T_3
	=
	\frac12r
	+O(r^{5/2}).
	\]
	Therefore,
	\[
	B(r)
	=
	\frac56r
	+O(r^{5/2})
	\sim r
	=
	\frac{|k_1|}{|k_2|},
	\]
	which completes the proof.
\end{proof}

From Proposition \ref{collision kernel}, we can directly derive 
\begin{equation}\label{R}
	|V_{kk_1k_2}|^2\lesssim |k|\min\lbrace |k_1|, |k_2|\rbrace^{7/2}.
\end{equation}
Similarly, the following estimates also hold
\begin{equation}\label{R_1}
	|V_{k_1kk_2}|^2\lesssim |k_1|\min\lbrace |k|, |k_2|\rbrace^{7/2},
\end{equation}
\begin{equation}\label{R_2}
	|V_{k_2kk_1}|^2\lesssim |k_2|\min\lbrace |k|, |k_1|\rbrace^{7/2}.
\end{equation}

\vspace{10pt}	
\section{Proof of main lemmas}
In this section, we will show $Q_{g,D}$ (\ref{dissi}) is dissipative, $Q_{g,b}$ (\ref{bounded}) and $\mathcal{R}_g$ (\ref{extra}) are bounded in some spaces.

First recall the definition of dissipative operator \cite[page~13-14]{pazy2012semigroups}:
\begin{defi}\label{def}
	Let \(X\) be a Banach space and \(X^{\ast}\) its dual space. For every \(x\in X\), we define the duality set \(F(x)\subseteq X^{\ast}\) by
	\begin{equation*}
		F(x):=\left\{x^{\ast}\in X^{\ast}\big|\; \langle x^{\ast}, x \rangle_{X^{\ast},X}=\|x\|^2=\|x^{\ast}\|^2\right\}.
	\end{equation*}
	A linear operator \(A\) is called dissipative if, for every \(x\in D(A)\), there exists \(x^{\ast}\in F(x)\) such that
	\begin{equation*}
		\Re e\langle x^{\ast}, Ax\rangle_{X^{\ast},X}\le 0.
	\end{equation*}
	Equivalently, \(A\) is dissipative if and only if
	\begin{equation*}
		\|(\lambda I-A)x\|\ge \lambda\|x\|, \quad \forall x\in D(A), \; \lambda>0.
	\end{equation*}
\end{defi}

\vspace{4pt}
\begin{lemma}\label{dissipative lemma}
	$Q_{g,D}$ is dissipative for all $1\le p\le \infty$.
\end{lemma}
\begin{proof}
	By definition and density, it suffices to prove
\begin{equation}\label{d}
	\left(h^{p-1}, Q_{g,D}(t)h\right)_{L^2}\le 0, \quad \forall h\in \mathcal{S}(\R^d),\; h\ge 0.
\end{equation}
Plugging (\ref{dissi}) into (\ref{d}), we can obtain
\begin{equation}\label{1}
	\iiint_{\R^{3d}} 4\pi
	|V_{k k_1 k_2}|^2
	\delta(k-k_1-k_2)
	\delta(\omega-\omega_{1}-\omega_{2}) \left[2g_1h_2 \chi_{2\ge 1}-g_1h-g_2h\right]h^{p-1}\; dkdk_1dk_2,
\end{equation}
\begin{equation}\label{2}
	-\iiint_{\R^{3d}} 4\pi
	|V_{k_1 k k_2}|^2
	\delta(k_1-k-k_2)
	\delta(\omega_{1}-\omega-\omega_{2}) \left[g_2h-g_2h_1\right]h^{p-1}\; dkdk_1dk_2,
\end{equation}

\begin{equation}\label{3}
	-\iiint_{\R^{3d}} 4\pi
	|V_{k_2 k k_1}|^2
	\delta(k_2-k-k_1)
	\delta(\omega_{2}-\omega-\omega_{1}) \left[g_1h-g_1h_2\right]h^{p-1}\; dkdk_1dk_2.
\end{equation}
Note that $V_{kk_1k_2}=V_{kk_2k_1}$; we can exchange the variables
\begin{equation*}
	k_1\leftrightarrow k_2,
\end{equation*}
then (\ref{1}) can be rewritten as 
\begin{equation*}
	(\ref{1})=\iiint_{\R^{3d}} 4\pi
	|V_{k k_1 k_2}|^2
	\delta(k-k_1-k_2)
	\delta(\omega-\omega_{1}-\omega_{2}) 
\end{equation*}
\begin{equation*}
	\times \left[g_1h_2\chi_{2\ge1}+g_2h_1\chi_{1\ge2}-g_1h-g_2h\right]h^{p-1} \; dkdk_1dk_2.
\end{equation*}

Similarly, one can swap the variables $k_1\leftrightarrow k$, $k_2\leftrightarrow k$ in (\ref{2}) and (\ref{3}), respectively. 
\begin{equation*}
	(\ref{2})=-\iiint_{\R^{3d}} 4\pi
	|V_{k k_1 k_2}|^2
	\delta(k-k_1-k_2)
	\delta(\omega-\omega_1-\omega_{2}) \left[g_2h_1-g_2h\right]h_1^{p-1}\; dkdk_1dk_2,
\end{equation*}
\begin{equation*}
	(\ref{3})=-\iiint_{\R^{3d}} 4\pi
	|V_{k k_1 k_2}|^2
	\delta(k-k_1-k_2)
	\delta(\omega-\omega_1-\omega_{2}) \left[g_1h_2-g_1h\right]h_2^{p-1}\; dkdk_1dk_2.
\end{equation*}
	
	Now, rearranging the above terms, we derive
\begin{equation*}
	\left(h^{p-1}, Q_{g,D}(t)h\right)_{L^2}=\iiint_{\R^{3d}} 4\pi
	|V_{k k_1 k_2}|^2
	\delta(k-k_1-k_2)
	\delta(\omega-\omega_{1}-\omega_{2}) \Gamma(k,k_1,k_2) dkdk_1dk_2,
\end{equation*}
where $\Gamma(k,k_1,k_2)$ is given by 
\begin{equation*}
	\Gamma(k,k_1,k_2):=\left(g_1h_2 h^{p-1}\chi_{2\ge 1}-\frac{1}{p}g_1 h_2^{p}-\frac{p-1}{p}g_1 h^p\right)
\end{equation*}
\begin{equation*}
	+\left(g_2h_1 h^{p-1}\chi_{1\ge 2}-\frac{1}{p}g_2 h_1^{p}-\frac{p-1}{p}g_2 h^p\right)
\end{equation*}
\begin{equation*}
	+\left(g_2hh_1^{p-1}-\frac{1}{p}g_2 h^p-\frac{p-1}{p}g_2 h_1^{p}\right)
\end{equation*}
\begin{equation*}
	+\left(g_1hh_2^{p-1}-\frac{1}{p}g_1h^{p}-\frac{p-1}{p}g_1h_2^p\right).
\end{equation*}
By Young's inequality, every term in the bracket is non-positive, which ensures
\begin{equation*}
	\left(h^{p-1}, Q_{g,D}(t)h\right)_{L^2}\le0.
\end{equation*}
\end{proof}

To establish the boundedness, we first prove the following estimate for the collision integral. The proof is similar to that of Lemma~3.14 in \cite{pan2026local}.

\begin{prop}\label{3.14}
	Define the multidimensional collision integral as follows:
	\begin{equation}\label{3.141}
		G_{F}^1(k):=\iint_{\R^{2d}}
		\delta(k-k_1-k_2)
		\delta(\omega-\omega_{1}-\omega_{2}) F(k,k_1,k_2) \; dk_1dk_2,
	\end{equation}
	\begin{equation}\label{3.142}
		G_{F}^2(k):=\iint_{\R^{2d}}
		\delta(k_1-k-k_2)
		\delta(\omega_1-\omega-\omega_{2}) F(k,k_1,k_2) \; dk_1dk_2,
	\end{equation}
	\begin{equation}\label{3.143}
		G_{F}^3(k):=\iint_{\R^{2d}}
		\delta(k_2-k-k_1)
		\delta(\omega_2-\omega-\omega_{1}) F(k,k_1,k_2) \; dk_1dk_2.
	\end{equation}
	Then the following estimates hold
	\begin{equation}\label{3.1411}
		|G_F^1(k)|\lesssim \|\min\lbrace \langle k_1\rangle^{d+1}, \langle k_2\rangle^{d+1}\rbrace F(k,k_1,k_2)\rbrace\|_{L_{k_1,k_2}^{\infty}},
	\end{equation}
	\begin{equation}\label{3.1421}
		|G_{F\chi_{0\ge 2}}^2(k)|\lesssim \|\langle k_2\rangle^{d+1} F(k,k_1,k_2)\rbrace\|_{L_{k_1,k_2}^{\infty}},
	\end{equation}
	\begin{equation}\label{3.1431}
		|G_{F\chi_{0\ge 1}}^3(k)|\lesssim \|\langle k_1\rangle^{d+1} F(k,k_1,k_2)\rbrace\|_{L_{k_1,k_2}^{\infty}}.
	\end{equation}
\end{prop}
\begin{proof}
	For (\ref{3.141}), we can split the integral into two parts:
	\begin{equation*}
		G_{F}^1(k)=\iint_{\R^{2d}}
		\delta(k-k_1-k_2)
		\delta(\omega-\omega_{1}-\omega_{2}) \chi_{1\ge 2} F(k,k_1,k_2) \; dk_1dk_2
	\end{equation*}
	\begin{equation*}
		+\iint_{\R^{2d}}
		\delta(k-k_1-k_2)
		\delta(\omega-\omega_{1}-\omega_{2}) \chi_{2>1} F(k,k_1,k_2) \; dk_1dk_2.
	\end{equation*}
	By symmetry, it suffices to consider the first integral. Integrating with respect to \(k_1\) first, we obtain
	\begin{equation*}
		\iint_{\R^{2d}}
		\delta(k-k_1-k_2)
		\delta(\omega-\omega_{1}-\omega_{2}) \chi_{1\ge 2} F(k,k_1,k_2) \; dk_1dk_2
	\end{equation*}
	\begin{equation}\label{6.271}
		=\int_{\R^{d}}
		\delta(\Delta \omega(k,k_2)) \chi_{1\ge 2} F(k,k_1,k_2)|_{k_1=k-k_2} \; dk_2,
	\end{equation}
	where $\Delta \omega(k,k_2)$ denotes
	\begin{equation*}
		\Delta \omega(k,k_2):=|k|^{3/2}-|k-k_2|^{3/2}-|k_2|^{3/2}.
	\end{equation*}
	Simple calculation shows that 
	\begin{equation}\label{c_1}
		\big| \nabla_{k_2} \Delta \omega(k,k_2)\big|=\frac{3}{2}\bigg|\frac{k_1}{|k_1|^{1/2}}-\frac{k_2}{|k_2|^{1/2}}\bigg|\ge \frac{3}{2}\frac{|k_1-k_2|}{|k_1|^{1/2}+|k_2|^{1/2}}.
	\end{equation}
	Moreover, by parallelogram law, we also have
	\begin{equation}\label{c_2}
		|k_1-k_2|\ge \sqrt{2^{2/3}-1}\cdot|k_1+k_2|=\sqrt{2^{2/3}-1}\cdot|k|,
	\end{equation}
	which implies
	\begin{equation}\label{6.27}
		\big| \nabla_{k_2} \Delta \omega(k,k_2)\big|\gtrsim |k|^{1/2}.
	\end{equation}
	Next, we decompose the integration domain according to the direction of the vector
	\[
	V(k_2):=\nabla_{k_2}\Delta\omega(k,k_2).
	\]
	Choose an open covering of the unit sphere \(S^{d-1}\) by finitely overlapping angular sectors
	\(\{C_j\}_{j=1}^N\), where \(N=N(d)\). For each cone \(C_j\), fix a unit vector
	\(e_j\) such that
	\[
	|e_j\cdot V(k_2)|
	\sim |V(k_2)|
	\]
	whenever \(V(k_2)\in C_j\).
	
	Let \(\{\phi_j\}_{j=1}^N\) be a smooth partition of unity subordinate to this covering,
	so that
	\[
	\sum_{j=1}^N\phi_j(V(k_2))=1,
	\qquad
	\forall\,k_2\in \R^d.
	\]
	
	Accordingly,
	\begin{equation*}
		(\ref{6.271})=\sum_{j=1}^N\int_{\mathbb R^{d}}
		\delta(\Delta\omega(k,k_2))
		\chi_{1\ge 2}\phi_j(V(k_2))
		F(k,k_2+k_3-k,k_2,k_3)\; dk_2.
	\end{equation*}
	For each localized piece, we introduce the coordinates
	\[
	k_2=se_j+t,
	\]
	where \(s\in\mathbb R\) denotes the coordinate in the \(e_j\)-direction, and
	\(t\in\mathbb R^{d-1}\) is the corresponding transverse variable.
	Applying the one-dimensional property of the Dirac function with respect to the \(s\)-variable together with (\ref{6.27}), we obtain
	\begin{equation*}
		|(\ref{6.271})|\lesssim \sum_{j=1}^N \int_{\R^{d-1}}\frac{1}{\langle t\rangle^{d}|t|^{1/2}}\sum_{s\in Z_t}\|\langle k_2\rangle^{d}F(k,k_1,k_2)|_{k_2=se_j+t}\|_{L_{k_1}^{\infty}}dt\lesssim \|\langle k_2\rangle^{d}F(k,k_1,k_2)\|_{L_{k_1, k_2}^{\infty}},
	\end{equation*}
	where we used $|k|\ge |k_2|\gtrsim |t|$ and 
\[
Z_t
=
\left\{
s\in\R:
\Delta\omega(k,k_2)\big|_{k_2=se_j+t}=0
\right\}
\]
denotes the set of roots of the resonance equation. By the convexity of the dispersion relation, one can directly show that the cardinality \(|Z_t|\) is uniformly bounded by a constant depending only on the dimension.
	
	Similarly, for (\ref{3.1421}), (\ref{3.1431}), we integrate $k_1$, $k_2$ first, respectively. Take (\ref{3.1421}) for example; one can obtain
	\begin{equation*}
		G_{F\chi_{0\ge 2}}^2(k)
		=\int_{\R^{d}}
		\delta(\Delta \widetilde{\omega}(k,k_2)) \chi_{0\ge 2} F(k,k_1,k_2)|_{k_1=k+k_2} \; dk_2,
	\end{equation*}
	where $\Delta \widetilde{\omega}(k,k_2)$ denotes
	\begin{equation*}
		\Delta \widetilde{\omega}(k,k_2):=|k+k_2|^{3/2}-|k|^{3/2}-|k_2|^{3/2}.
	\end{equation*}
	Moreover, we also have 
	\begin{equation*}
		\big|\nabla_{k_2}\widetilde{\omega}(k,k_2)\big|\gtrsim |k|^{1/2}\ge |k_2|^{1/2}.
	\end{equation*}
	Then, following the same procedure as in the first estimate, one can obtain
	\begin{equation*}
		|G_{F\chi_{0\ge 2}}^2(k)|\lesssim \|\langle k_2\rangle^{d+1} F(k,k_1,k_2)\rbrace\|_{L_{k_1,k_2}^{\infty}}.
	\end{equation*}
\end{proof}
\begin{rem}\label{a}
	From the proof above, we can also derive some variants of (\ref{3.1421}), (\ref{3.1431}). In fact, the following estimates hold:
	\begin{equation}\label{3.14211}
		|G_{F}^2(k)|\lesssim \|\langle k_1\rangle^{d+1} F(k,k_1,k_2)\rbrace\|_{L_{k_1,k_2}^{\infty}},
	\end{equation}
	\begin{equation}\label{3.14311}
		|G_{F}^3(k)|\lesssim \|\langle k_2\rangle^{d+1} F(k,k_1,k_2)\rbrace\|_{L_{k_1,k_2}^{\infty}}.
	\end{equation}

	And for detailed proofs of (\ref{c_1}), (\ref{c_2}) and $|Z_t|=O_d(1)$, we refer to Lemma \ref{c_11} and Lemma \ref{c_22} in the Appendix A. 
\end{rem}

\vspace{7pt}
Now, we are ready to apply Proposition~\ref{3.14} to establish the boundedness of \(Q_{g,b}\).
\begin{lemma}\label{boundedness of Qb}
	$Q_{g,b}(t)$ is a bounded operator on $L_{s}^{p}(\R^d)$, for all $1\le p\le \infty$, $s\ge 0$, satisfying
	\begin{equation*}
		\sup_{t\in [0,T]}\|Q_{g,b}(t)\|_{p,s\to p,s}\lesssim_d \sup_{t\in[0,T]}\|g(t)\|_{\infty,2d+10}.
	\end{equation*}
\end{lemma}
\begin{proof}
	Define the multiplier $\gamma_i(k)$, $i=1,2$:
	\begin{equation*}
		\gamma_1(k):=\iint_{\R^{2d}} 4\pi
		|V_{k_1 k k_2}|^2
		\delta(k_1-k-k_2)
		\delta(\omega_{1}-\omega-\omega_{2}) g_1 \; dk_1dk_2,
	\end{equation*}
	\begin{equation*}
		\gamma_2(k):=\iint_{\R^{2d}} 4\pi
		|V_{k_2 k k_1}|^2
		\delta(k_2-k-k_1)
		\delta(\omega_{2}-\omega-\omega_{1}) g_2 \; dk_1dk_2.
	\end{equation*}
	We can write 
	\begin{equation*}
		(Q_{g,b}(t)h)(k)=(\gamma_1(k)+\gamma_2(k))h(k), \quad \forall h\in L_{s}^{p}(\R^d).
	\end{equation*}
	Then it suffices to show
	\begin{equation*}
		\sup_{t\in [0,T]}\|\gamma_1\|_{\infty}, \sup_{t\in [0,T]}\|\gamma_2\|_{\infty}\lesssim_d \sup_{t\in[0,T]}\|g(t)\|_{\infty,2d+10}.
	\end{equation*}
	Combining (\ref{R_1}) with (\ref{3.14211}), we can obtain
	\begin{equation*}
		\|\gamma_1\|_{\infty}\lesssim \iint_{\R^{2d}} 
		\delta(k_1-k-k_2)
		\delta(\omega_{1}-\omega-\omega_{2}) |k_1|^{9/2}g_1 \; dk_1dk_2
	\end{equation*}
	\begin{equation*}
		\lesssim \|\langle k_1\rangle^{d+11/2}g(k_1)\|_{L_{k_1}^{\infty}}\le \|g(t)\|_{\infty,2d+10}.
	\end{equation*}
	Similarly, applying (\ref{R_2}) and (\ref{3.14311}), one can check that
	\begin{equation*}
		\|\gamma_2\|_{\infty}\lesssim \|\langle k_2\rangle^{d+11/2}g(k_2)\|_{L_{k_2}^{\infty}}\le \|g(t)\|_{\infty,2d+10}.
	\end{equation*}
\end{proof}

\vspace{7pt}
To establish the boundedness of the operator \(\mathcal{R}_g(t)\), we further introduce the notation
\begin{equation*}
	\mathcal{R}_{g}(t) \varphi= \iint_{\R^{2d}} 8\pi
	|V_{k k_1 k_2}|^2
	\delta(k-k_1-k_2)
	\delta(\omega-\omega_{1}-\omega_{2}) g_1 \dfrac{\langle k\rangle^{a}-\langle k_2\rangle^{a}}{\langle k_2\rangle^{a}} \varphi_2 \chi_{2\ge 1} \; dk_1dk_2
\end{equation*}
\begin{equation*}
	+\iint_{\R^{2d}} 4\pi
	|V_{k_1 k k_2}|^2
	\delta(k_1-k-k_2)
	\delta(\omega_1-\omega-\omega_{2}) g_2 \dfrac{\langle k\rangle^{a}-\langle k_1\rangle^{a}}{\langle k_1\rangle^{a}} \varphi_1 \; dk_1dk_2
\end{equation*}
\begin{equation*}
	+\iint_{\R^{2d}} 4\pi
	|V_{k_2 k k_1}|^2
	\delta(k_2-k-k_1)
	\delta(\omega_2-\omega-\omega_{1}) g_1 \dfrac{\langle k\rangle^{a}-\langle k_2\rangle^{a}}{\langle k_2\rangle^{a}} \varphi_2 \; dk_1dk_2
\end{equation*}
\begin{equation*}
	=:\mathcal{R}_{g}^1(t) \varphi+\mathcal{R}_{g}^2(t) \varphi+\mathcal{R}_{g}^3(t) \varphi.
\end{equation*}

\vspace{5pt}
\begin{lemma}\label{boundedness of R}
	$\mathcal{R}_{g}(t)$ is bounded on $L^{p}(\R^d)$ for all $1\le p\le \infty$, satisfying
	\begin{equation*}
		\|\mathcal{R}_{g}(t)\|_{p\to p}\lesssim_d \|g\|_{\infty,2d+10}.
	\end{equation*}
\end{lemma}
\begin{proof}
	It reduces to show 
	\begin{equation*}
		\|\mathcal{R}_{g}^i(t)\|_{p\to p}\lesssim_d \|g\|_{\infty,2d+10}, \quad \forall i=1,2,3.
	\end{equation*}
	By interpolation, we only need to deal with $p=1,\infty$.
	
	Note that, for $\mathcal{R}_g^1$, we have $|k|\sim |k_2|\ge |k_1|$. Then applying mean value theorem, one can see
	\begin{equation}\label{s'}
		\Bigg|\frac{\langle k\rangle^{a}-\langle k_2\rangle^{a}}{\langle k_2\rangle^{a}}\Bigg|\lesssim_a \frac{\big||k|-|k_2|\big|\langle k_2\rangle^{a-1}}{\langle k_2\rangle^{a}}\lesssim  \frac{|k_1|}{\langle k\rangle}.
	\end{equation}
	Then combining (\ref{s'}) with (\ref{R}), (\ref{3.1411}), one can obtain
	\begin{equation*}
		|\mathcal{R}_g^1\varphi(k)|\lesssim \Bigg|\iint_{\R^{2d}} 
		\delta(k-k_1-k_2)
		\delta(\omega-\omega_{1}-\omega_{2}) |k_1|^{9/2}g_1 \varphi_2 \chi_{2\ge 1} \; dk_1dk_2\Bigg|
	\end{equation*}
	\begin{equation*}
		\lesssim \|\langle k_1\rangle^{d+11/2} g(k_1)\|_{L_{k_1}^{\infty}}\|\varphi\|_{\infty}\le \|g\|_{\infty, 2d+10}\|\varphi\|_{\infty}.
	\end{equation*}
	Thus, we have verified the case $p=\infty$. 
	
	Now, for $\varphi\in L^1(\R^d)$, we can similarly derive
	\begin{equation*}
		\int_{\R^d}|\mathcal{R}_g^1\varphi(k)|dk\lesssim \iiint_{\R^{3d}} 
		\delta(k-k_1-k_2)
		\delta(\omega-\omega_{1}-\omega_{2}) |k_1|^{9/2}g_1 |\varphi_2| \chi_{2\ge 1} \; dkdk_1dk_2
	\end{equation*}
	\begin{equation*}
		=\int_{\R^d} |\varphi_2| dk_2\iint_{\R^{2d}}\delta(k-k_1-k_2)
		\delta(\omega-\omega_{1}-\omega_{2}) |k_1|^{9/2}g_1 \chi_{2\ge 1} \; dkdk_1
	\end{equation*}
	\begin{equation}\label{sss}
		\lesssim \|\varphi\|_1 \sup_{k_2}\left\{\iint_{\R^{2d}}\delta(k-k_1-k_2)
		\delta(\omega-\omega_{1}-\omega_{2}) |k_1|^{9/2}g_1 \chi_{2\ge 1} \; dkdk_1\right\}.
	\end{equation}
	For the integral in the bracket, we first integrate on $k$, which yields
	\begin{equation*}
		\iint_{\R^{2d}}\delta(k-k_1-k_2)
		\delta(\omega-\omega_{1}-\omega_{2}) |k_1|^{9/2}g_1 \chi_{2\ge 1} \; dkdk_1
	\end{equation*}
	\begin{equation}\label{ssss}
		=\int_{\R^d} 
		\delta(\Delta\omega(k_1,k_2)) |k_1|^{9/2}g_1 \chi_{2\ge 1} dk_1,
	\end{equation}
	where $\Delta\omega(k_1,k_2):=|k_1+k_2|^{3/2}-|k_1|^{3/2}-|k_2|^{3/2}$.
	
A simple calculation shows that
\begin{equation}\label{sssss}
	\Big| \nabla_{k_1}\Delta\omega(k_1,k_2)\Big|=\frac{3}{2}\bigg|\frac{k_1+k_2}{|k_1+k_2|^{1/2}}-\frac{k_1}{|k_1|^{1/2}}\bigg|\gtrsim |k_2|^{1/2}\ge |k_1|^{1/2}.
\end{equation}
We can then follow the same argument as in Proposition~\ref{3.14} to obtain
\begin{equation*}
	\int_{\R^d} 
	\delta(\Delta\omega(k_1,k_2)) |k_1|^{9/2}g_1 \chi_{2\ge 1} dk_1\lesssim \|\langle k_1\rangle^{d+5}g(k_1)\|_{L_{k_1}^{\infty}}\le \|g\|_{\infty, 2d+10},
\end{equation*}
which establishes the \(L^{1}\)-boundedness of the operator \(\mathcal{R}_g^1(t)\).

Next, we establish the boundedness of \(\mathcal{R}_g^2(t)\) and \(\mathcal{R}_g^3(t)\). By symmetry, it suffices to consider \(\mathcal{R}_g^2(t)\).
	
	Splitting $\mathcal{R}_g^2(t)$ into two parts:
	\begin{equation*}
		\mathcal{R}_g^2(t)\varphi=\iint_{\R^{2d}} 4\pi
		|V_{k_1 k k_2}|^2
		\delta(k_1-k-k_2)
		\delta(\omega_1-\omega-\omega_{2}) g_2 \dfrac{\langle k\rangle^{a}-\langle k_1\rangle^{a}}{\langle k_1\rangle^{a}} \varphi_1 \chi_{0\ge 2} \; dk_1dk_2
	\end{equation*}
	\begin{equation*}
		+\iint_{\R^{2d}} 4\pi
		|V_{k_1 k k_2}|^2
		\delta(k_1-k-k_2)
		\delta(\omega_1-\omega-\omega_{2}) g_2 \dfrac{\langle k\rangle^{a}-\langle k_1\rangle^{a}}{\langle k_1\rangle^{a}} \varphi_1 \chi_{2>0} \; dk_1dk_2
	\end{equation*}
	\begin{equation*}
		=: \mathcal{R}_g^{21}(t)\varphi+\mathcal{R}_g^{22}(t)\varphi.
	\end{equation*}
	Note that, by mean value theorem, we also have 
	\begin{equation}\label{ss}
		\bigg|\dfrac{\langle k\rangle^{a}-\langle k_1\rangle^{a}}{\langle k_1\rangle^{a}}\bigg|\lesssim \frac{|k_2|}{\langle k_1\rangle},
	\end{equation}
	when $|k|\ge |k_2|$.
	
	Then applying (\ref{3.1421}) together with (\ref{R_1}) and (\ref{ss}), one can derive the $L^\infty$-boundedness
	\begin{equation*}
		|\mathcal{R}_g^{21}(t)\varphi(k)|\lesssim \|\langle k_2\rangle^{d+11/2} g(k_2)\|_{L_{k_2}^{\infty}}\|\varphi\|_{\infty}\le \|g\|_{\infty, 2d+10}\|\varphi\|_{\infty}.
	\end{equation*}

Following a similar procedure as in (\ref{sss})-(\ref{sssss}), the $L^1$-boundedness also holds:
\begin{equation*}
	\int_{\R^d}|\mathcal{R}_g^{21}(t)\varphi(k)|dk \lesssim \|\varphi\|_1\|\langle k_2\rangle^{d+5}g(k_2)\|_{L_{k_2}^{\infty}}\le \|\varphi\|_1\|g\|_{\infty, 2d+10}.
\end{equation*}

Finally, for $\mathcal{R}_g^{22}(t)$, the following trivial bound holds:
\begin{equation*}
	\bigg|\dfrac{\langle k\rangle^{a}-\langle k_1\rangle^{a}}{\langle k_1\rangle^{a}}\bigg|\lesssim 1.
\end{equation*}
Since the integrand is restricted to $\lbrace |k_1|\sim |k_2|\ge |k|\rbrace$, any such factors can be absorbed into $g_2$. One can then easily derive
\begin{equation*}
	|\mathcal{R}_g^{22}(t)\varphi(k)|\lesssim \|\langle k_2\rangle^{d+6} g(k_2)\|_{L_{k_2}^{\infty}}\|\varphi\|_{\infty}\le \|g\|_{\infty, 2d+10}\|\varphi\|_{\infty}, 
\end{equation*}
\begin{equation*}
	\int_{\R^d}|\mathcal{R}_g^{22}(t)\varphi(k)|dk \lesssim \|\varphi\|_1\|\langle k_2\rangle^{d+6}g(k_2)\|_{L_{k_2}^{\infty}}\le \|\varphi\|_1\|g\|_{\infty, 2d+10}.
\end{equation*}
Thus, we have completed the proof of the $L^p$-boundedness of $\mathcal{R}_g(t)$.

\end{proof}

\vspace{10pt}
\section{Energy estimate}
With the lemmas established in the previous section, we can now derive a key energy estimate for the iteration scheme (\ref{iter}).

Recall that the weighted solution $\langle k\rangle^{a}n_{\ell+1} $ satisfies
\begin{equation*}
	\begin{cases}
		\dfrac{\partial }{\partial t}(\langle k\rangle^{a} n_{\ell+1})=[Q_{\ell}(t)+\mathcal{R}_{\ell}(t)](\langle k\rangle^{a} n_{\ell+1}),  \\[4pt]
		\langle k\rangle^{a} n_{\ell+1}(0,k)=\langle k\rangle^{a}n_0(k), \qquad (t,k)\in \R^{+} \times \R^d.
	\end{cases}
\end{equation*}

We also denote $f_\ell:=\langle k\rangle^a n_{\ell+1}$ and denote the propagator of $[Q_{\ell}(t)+\mathcal{R}_{\ell}(t)]$ by $U_{\ell}(t,s)$ for $0\le s\le t$. Then we have $f_\ell=U_{\ell}(t,0)(\langle k\rangle^{a}n_0)$.

\begin{rem}\label{approx}
It is worth mentioning that the propagator $U_\ell(t,s)$ should be understood as a strong limit of a family of regularized operators. In fact, the dissipativity of $Q_{\ell,D}$ alone is not sufficient to guarantee the existence of a unique semigroup. Detailed discussions are deferred to Appendix B.
\end{rem}

\begin{lemma}\label{energy estimate}
	For $T>0$, $a\ge0$, $1\le p\le \infty$, we have the following energy estimate:
	\begin{equation*}
		\|n_{\ell+1}(t)\|_{p,a}=\|f_\ell(t)\|_{p}\le\exp \left(C_a t\cdot \sup_{t\in [0,T]}\|n_{\ell}(t)\|_{\infty, 2d+10}\right)\|n_0\|_{p,a},
	\end{equation*}
	where $C_a>0$ depends only on $a,d$.
	
\end{lemma} 
\begin{proof}
	Taking the time derivative of the $L^p$-norm, one can obtain
	\begin{equation*}
		\frac{1}{p}\frac{d}{dt}\|f_{\ell}(t)\|_p^p=\left(\partial_t f_{\ell}(t), f_{\ell}^{p-1}(t)\right)_{L^2}=\left(Q_{\ell}(t) f_{\ell}, f_{\ell}^{p-1}(t)\right)_{L^2}+\left(\mathcal{R}_\ell(t) f_{\ell}, f_{\ell}^{p-1}(t)\right)_{L^2}
	\end{equation*}
	\begin{equation*}
		=\left(Q_{\ell,D}(t) f_{\ell}, f_{\ell}^{p-1}(t)\right)_{L^2}+\left(Q_{\ell,b}(t) f_{\ell}, f_{\ell}^{p-1}(t)\right)_{L^2}+\left(\mathcal{R}_\ell(t) f_{\ell}, f_{\ell}^{p-1}(t)\right)_{L^2}.
	\end{equation*}
	From Lemma \ref{dissipative lemma}, we know $Q_{\ell,D}$ is dissipative, which implies
	\begin{equation*}
		\left(Q_{\ell,D}(t) f_{\ell}, f_{\ell}^{p-1}(t)\right)_{L^2}\le0.
	\end{equation*}
	Applying Lemma \ref{boundedness of Qb} and Lemma \ref{boundedness of R}, i.e., the boundedness of $Q_{\ell,b}$ and $\mathcal{R}_\ell$, one can derive
	\begin{equation*}
		\bigg|\left(Q_{\ell,b}(t) f_{\ell}, f_{\ell}^{p-1}(t)\right)_{L^2}\bigg|\lesssim \sup_{t\in [0,T]}\|n_{\ell}(t)\|_{\infty, 2d+10}\|f_\ell\|_p\|f_\ell^{p-1}\|_{p'}
	\end{equation*}
	\begin{equation*}
		=\sup_{t\in [0,T]}\|n_{\ell}(t)\|_{\infty, 2d+10}\|f_\ell\|_p^p,
	\end{equation*}
	\begin{equation*}
		\bigg|\left(\mathcal{R}_{\ell}(t) f_{\ell}, f_{\ell}^{p-1}(t)\right)_{L^2}\bigg|\lesssim \sup_{t\in [0,T]}\|n_{\ell}(t)\|_{\infty, 2d+10}\|f_\ell\|_p\|f_\ell^{p-1}\|_{p'}
	\end{equation*}
	\begin{equation*}
		=\sup_{t\in [0,T]}\|n_{\ell}(t)\|_{\infty, 2d+10}\|f_\ell\|_p^p,
	\end{equation*}
	where $p'=\frac{p}{p-1}$.
	
	Thus, we have the following differential inequality:
	\begin{equation*}
		\frac{d}{dt}\|f_{\ell}(t)\|_p\lesssim \left(\sup_{t\in [0,T]}\|n_{\ell}(t)\|_{\infty, 2d+10}\right)\|f_\ell(t)\|_p.
	\end{equation*}
	Applying Gronwall's inequality, we obtain the desired energy estimate.
\end{proof}
\begin{rem}\label{t-s}
	Similarly, we can also derive a slightly more general energy estimate:
	\begin{equation}\label{j}
		\|U_{\ell}(t,s)h\|_{p}\le\exp \left(C_a (t-s)\cdot \sup_{t\in [0,T]}\|n_{\ell}(t)\|_{\infty, 2d+10}\right)\|h\|_{p}, \quad \forall h\in L^p(\R^d),
	\end{equation}
	where $0\le s\le t$.
\end{rem}

As a direct corollary, the sequence $\{n_{\ell}\}_\ell$ is uniformly bounded in some space.
\begin{coro}\label{uniform}
	Assume the initial data $n_0\in L_{4d+12}^{\infty}(\R^d)$, then there exists 
	\begin{equation*}
		T_0=T_0\left(\|n_0\|_{\infty, 4d+12}\right)>0,
	\end{equation*}
	such that the following uniform bound holds:
	\begin{equation}\label{l}
		\sup_{t\in [0,T]}\|n_\ell(t)\|_{\infty, 4d+12}\le 2\|n_0\|_{\infty, 4d+12}, \quad \forall \ell\ge 1.
	\end{equation}
\end{coro}
\begin{proof}
	Applying Lemma \ref{energy estimate} with $a=4d+12$ and $p=\infty$, we can derive the energy estimate
	\begin{equation}\label{ll}
		\sup_{t\in [0,T]}\|n_{\ell+1}(t)\|_{\infty}\le\exp \left(C_a T\cdot \sup_{t\in [0,T]}\|n_{\ell}(t)\|_{\infty, 2d+10}\right)\|n_0\|_{\infty,4d+12}.
	\end{equation}
	From definition, $n_{1}(t,k)\equiv n_0(k)$, then (\ref{l}) holds trivially when $\ell=1$.
	
	Suppose that the induction hypothesis
	\begin{equation*}
		\sup_{t\in [0,T]}\|n_\ell(t)\|_{\infty, 4d+12}\le 2\|n_0\|_{\infty, 4d+12}
	\end{equation*}
	holds for $\ell=k$. Letting $\ell=k$ in (\ref{ll}), we obtain
	\begin{equation*}
		\sup_{t\in [0,T]}\|n_{\ell+1}(t)\|_{\infty}\le\exp \left(2C_a T\|n_0\|_{\infty, 4d+12}\right)\|n_0\|_{\infty,4d+12}.
	\end{equation*}
	By taking $T$ sufficiently small such that
	\begin{equation*}
		0<T\le \frac{1}{100C_a \|n_0\|_{\infty,4d+12}},
	\end{equation*}
	we can close the induction.
\end{proof}

\vspace{10pt}
\section{Proof of local existence}
To show that the sequence $\lbrace n_{\ell}\rbrace_\ell$ in the iteration scheme (\ref{iter}) converges in some space, we choose the weight $a=2(d+1)$ in (\ref{weighted equation}) and find that the difference $\langle k\rangle^{a}(n_{\ell+2}-n_{\ell+1})$ satisfies
\begin{equation*}
	\partial_t\left[\langle k\rangle^{a}(n_{\ell+2}-n_{\ell+1})\right]=\left(Q_{\ell+1}(t)+\mathcal{R}_{\ell+1}(t)\right)\left(\langle k\rangle^{a}(n_{\ell+2}-n_{\ell+1})\right)
\end{equation*}
\begin{equation*}
	+\left(Q_{\ell+1}(t)+\mathcal{R}_{\ell+1}(t)-Q_{\ell}(t)-\mathcal{R}_{\ell}(t)\right)\left(\langle k\rangle^{a}n_{\ell}\right)
\end{equation*}
\begin{equation}\label{k}
	=\left(Q_{\ell+1}(t)+\mathcal{R}_{\ell+1}(t)\right)\left(\langle k\rangle^{a}(n_{\ell+2}-n_{\ell+1})\right)+\mathcal{F}_{\ell}(t)\left(\langle k\rangle^{a}(n_{\ell+1}-n_{\ell})\right).
\end{equation}

\vspace{7pt}
Here we denote $g:=n_{\ell}$ and the operator $\mathcal{F}_{\ell}(t)$ acting on $H$ is defined as
\begin{equation*}
	\mathcal{F}_{\ell}(t)H:=\iint_{\R^{2d}}	4\pi
	|V_{k k_1 k_2}|^2
	\delta(k-k_1-k_2)
	\delta(\omega-\omega_{1}-\omega_{2})
\end{equation*}
\begin{equation*}
	\times \left\{2\frac{\langle k_2\rangle^{2(d+1)}}{\langle k_1\rangle^{2(d+1)}}g_2\chi_{2\ge1}H_1-\frac{\langle k\rangle^{2(d+1)}}{\langle k_1\rangle^{2(d+1)}}gH_1-\frac{\langle k\rangle^{2(d+1)}}{\langle k_2\rangle^{2(d+1)}}gH_2+2\frac{\langle k\rangle^{2(d+1)}-\langle k_2\rangle^{2(d+1)}}{\langle k_1\rangle^{2(d+1)}}g_2\chi_{2\ge1}H_1\right\}dk_1dk_2,
\end{equation*}

\vspace{4pt}
\begin{equation*}
	+\iint_{\R^{2d}}	4\pi
	|V_{k_1 k k_2}|^2
	\delta(k_1-k-k_2)
	\delta(\omega_{1}-\omega-\omega_{2})
\end{equation*}
\begin{equation*}
	\times \left\{-\frac{\langle k\rangle^{2(d+1)}}{\langle k_2\rangle^{2(d+1)}}gH_2+\frac{\langle k\rangle^{2(d+1)}}{\langle k_1\rangle^{2(d+1)}}gH_1+\frac{\langle k_1\rangle^{2(d+1)}}{\langle k_2\rangle^{2(d+1)}}g_1H_2+\frac{\langle k\rangle^{2(d+1)}-\langle k_1\rangle^{2(d+1)}}{\langle k_2\rangle^{2(d+1)}}g_1H_2\right\}dk_1dk_2,
\end{equation*}

\vspace{4pt}
\begin{equation*}
	+\iint_{\R^{2d}}	4\pi
	|V_{k_2 k k_1}|^2
	\delta(k_2-k-k_1)
	\delta(\omega_{2}-\omega-\omega_{1})
\end{equation*}
\begin{equation*}
	\times \left\{-\frac{\langle k\rangle^{2(d+1)}}{\langle k_1\rangle^{2(d+1)}}gH_1+\frac{\langle k\rangle^{2(d+1)}}{\langle k_2\rangle^{2(d+1)}}gH_2+\frac{\langle k_2\rangle^{2(d+1)}}{\langle k_1\rangle^{2(d+1)}}g_2H_1+\frac{\langle k\rangle^{2(d+1)}-\langle k_2\rangle^{2(d+1)}}{\langle k_1\rangle^{2(d+1)}}g_2H_1\right\}dk_1dk_2
\end{equation*}
\vspace{4pt}
\begin{equation*}
	=:\mathcal{F}_{\ell}^1(t)H+\mathcal{F}_{\ell}^2(t)H+\mathcal{F}_{\ell}^3(t)H.
\end{equation*}

We first establish $L^{p}$-boundedness of the operator $\mathcal{F}_{\ell}(t)$.
\begin{lemma}\label{Fn}
	$\mathcal{F}_{\ell}(t)$ is bounded on $L^{p}(\R^d)$ for all $1\le p\le \infty$, satisfying
	\begin{equation*}
		\|\mathcal{F}_{\ell}(t)\|_{p\to p}\lesssim_d \|n_{\ell}\|_{\infty, 4d+12}.
	\end{equation*}
\end{lemma}
\begin{proof}
	It suffices to show for $p=1, \infty$,
	\begin{equation*}
		\|\mathcal{F}_{\ell}^{i}(t)\|_{p\to p}\lesssim_d \|n_{\ell}\|_{\infty, 4d+12}, \quad  \forall i=1,2,3.
	\end{equation*}
	For $\mathcal{F}_{\ell}^{1}(t)$, we observe that every term in the bracket only involves $g, g_2$. Since the constraint 
	\begin{equation*}
		\delta(k-k_1-k_2)
		\delta(\omega-\omega_{1}-\omega_{2}) \chi_{2\ge1}
	\end{equation*}
	ensures 
	\begin{equation*}
		|k|\sim |k_2|\ge |k_1|,
	\end{equation*}
	any factors can be absorbed into $g, g_2$. And the weight index $4d+12$ is chosen so that the integrand has enough decay. Then one can directly check that
	\begin{equation*}
		\|\mathcal{F}_{\ell}^{1}(t)\|_{1\to 1}, \|\mathcal{F}_{\ell}^{1}(t)\|_{\infty\to \infty}\lesssim_d \|n_{\ell}\|_{\infty, 4d+12}.
	\end{equation*}
	
	By symmetry, we only need to prove the boundedness of $\mathcal{F}_{\ell}^{2}(t)$. Note that on the constraint 
	\begin{equation*}
		\delta(k_1-k-k_2)
		\delta(\omega_{1}-\omega-\omega_{2}),
	\end{equation*}
	we have $|k_1|\ge \max\{|k|, |k_2|\}$. Then by the same reasoning, the terms
	\begin{equation*}
		\frac{\langle k_1\rangle^{2(d+1)}}{\langle k_2\rangle^{2(d+1)}}g_1H_2, \quad \frac{\langle k\rangle^{2(d+1)}-\langle k_1\rangle^{2(d+1)}}{\langle k_2\rangle^{2(d+1)}}g_1H_2,
	\end{equation*}
	can be controlled directly. And we are left with the terms
	\begin{equation*}
		-\frac{\langle k\rangle^{2(d+1)}}{\langle k_2\rangle^{2(d+1)}}gH_2, \quad \frac{\langle k\rangle^{2(d+1)}}{\langle k_1\rangle^{2(d+1)}}gH_1.
	\end{equation*}
	For the first term, one can further split it into two parts:
	\begin{equation*}
		-\frac{\langle k\rangle^{2(d+1)}}{\langle k_2\rangle^{2(d+1)}}gH_2=-\frac{\langle k\rangle^{2(d+1)}}{\langle k_2\rangle^{2(d+1)}}gH_2\chi_{0\ge 2}-\frac{\langle k\rangle^{2(d+1)}}{\langle k_2\rangle^{2(d+1)}}gH_2\chi_{2>0}.
	\end{equation*}
	Together with the extra restriction $\chi_{0\ge2}$, we see that $|k_1|\sim |k|$, and any such factors can again be absorbed into $g$. Under the restriction $\chi_{2>0}$, we have $|k_1|\sim |k_2|$. Also note that the weight index $2(d+1)$ is large enough so that 
	\begin{equation*}
		\langle k_2\rangle^{-2(d+1)}
	\end{equation*}
	guarantees the decay of the integrand. Thus, the corresponding integral operator is bounded.
	
	For the second term, we can follow the same argument since $|k_1|$ is always the largest.
	
\end{proof}

Now, we conclude the local existence part of Theorem \ref{main} by a standard contraction mapping argument.

\begin{proof}[Proof of Theorem \ref{main}]
	We first denote the quantity $E:=2\|n_0\|_{\infty,4d+12}$. Recall the equation (\ref{k}), with $a=2(d+1)$:
	\begin{equation*}
		\partial_t\left[\langle k\rangle^{a}(n_{\ell+2}-n_{\ell+1})\right]=\left(Q_{\ell+1}(t)+\mathcal{R}_{\ell+1}(t)\right)\left(\langle k\rangle^{a}(n_{\ell+2}-n_{\ell+1})\right)
	\end{equation*}
	\begin{equation*}
		+\mathcal{F}_{\ell}(t)\left(\langle k\rangle^{a}(n_{\ell+1}-n_{\ell})\right).
	\end{equation*}
	
	Then applying Duhamel's formula, the uniform bound (\ref{l}) and the energy estimate (\ref{j}), one can obtain
	\begin{equation*}
		\|n_{\ell+2}(t)-n_{\ell+1}(t)\|_{\infty,2(d+1)}\le \int_{0}^t\left\|U_\ell(t,s) \mathcal{F}_{\ell}(s) \langle k\rangle^{2(d+1)}\left(n_{\ell+1}-n_{\ell}\right)\right\|_2 ds
	\end{equation*}
	\begin{equation*}
		\lesssim \int_{0}^t \exp\left(C_a(t-s)E\right)E \|n_{\ell+1}(s)-n_{\ell}(s)\|_{\infty,2(d+1)}ds
	\end{equation*}
	\begin{equation*}
		\lesssim \left(e^{C_a tE}-1\right)\sup_{s\in[0,t]}\|n_{\ell+1}(s)-n_{\ell}(s)\|_{\infty,2(d+1)}.
	\end{equation*}
	Thus, by taking $T=T(E)>0$ sufficiently small, we verify the contraction
	\begin{equation*}
		\sup_{t\in[0,T]}\|n_{\ell+2}(t)-n_{\ell+1}(t)\|_{\infty,2(d+1)}\le \frac{1}{2}\sup_{t\in[0,T]}\|n_{\ell+1}(t)-n_{\ell}(t)\|_{\infty,2(d+1)}.
	\end{equation*}
	Then the sequence $\{n_{\ell}\}_\ell$ converges in the space $C([0,T];L_{2(d+1)}^\infty(\R^d))$ and its limit $n(t,k)$ is the strong solution to the capillary wave kinetic equation (\ref{Capillary}).
	
	By choosing a pointwise convergent subsequence, one can also ensure that
	\begin{equation*}
		\sup_{t\in [0,T]}\|n(t,k)\|_{\infty,4d+12}\le 2\|n_0\|_{\infty,4d+12}.
	\end{equation*}
\end{proof}

       \newpage
       \appendix
       \section{Auxiliary estimates}
      In Appendix A, we provide detailed proofs of some estimates that appeared in the above sections.
       
      \begin{lemma}\label{c_11}
      	Suppose that 
      	\begin{equation}\label{b}
      		k=k_1+k_2,\quad |k|^{3/2}=|k_1|^{3/2}+|k_2|^{3/2},
      	\end{equation}
      	then we have 
      	\begin{equation}\label{v}
      		\bigg|\frac{k_1}{|k_1|^{1/2}}-\frac{k_2}{|k_2|^{1/2}}\bigg|\ge \frac{|k_1-k_2|}{|k_1|^{1/2}+|k_2|^{1/2}},
      	\end{equation}
      	
      	\vspace{5pt}
      	\begin{equation}\label{vv}
      		|k_1-k_2|\ge \sqrt{2^{2/3}-1}|k_1+k_2|=\sqrt{2^{2/3}-1}|k|.
      	\end{equation}
      \end{lemma}
      \begin{proof}
      	Let $a:=|k_1|$, $b:=|k_2|$, and let the angle be
      	\begin{equation*}
      		\gamma:=\widehat{k_1}\cdot \widehat{k_2}\in [-1,1], \quad \widehat{k_1}:=\frac{k_1}{|k_1|},\; \widehat{k_2}:=\frac{k_2}{|k_2|}.
      	\end{equation*}
      	
      	Then we can write
      	\begin{equation*}
      		\bigg|\frac{k_1}{|k_1|^{1/2}}-\frac{k_2}{|k_2|^{1/2}}\bigg|^2=a+b-2\sqrt{ab}\gamma,
      	\end{equation*}
      	\vspace{5pt}
      	\begin{equation*}
      		\frac{|k_1-k_2|^2}{\left(|k_1|^{1/2}+|k_2|^{1/2}\right)^{2}}=\frac{a^2+b^2-2ab\gamma}{(\sqrt{a}+\sqrt{b})^2}.
      	\end{equation*}
      	
      	Note that the coefficients of $\gamma$ satisfy
      	\begin{equation*}
      		-2\sqrt{ab}\le \frac{-2ab}{(\sqrt{a}+\sqrt{b})^2}.
      	\end{equation*}
      	We then check the endpoint $\gamma=1$ and obtain inequality (\ref{v}).
      	
      	For (\ref{vv}), we combine (\ref{b}) with the parallelogram law
      	\begin{equation*}
      		|k_1-k_2|^2=2\left(|k_1|^2+|k_2|^2\right)-|k_1+k_2|^2=2\left(|k_1|^2+|k_2|^2\right)-|k|^2
      	\end{equation*}
      	\begin{equation*}
      		=2(a^2+b^2)-\left(a^{3/2}+b^{3/2}\right)^{4/3}.
      	\end{equation*}
      	Recall the power mean inequality:
      	\begin{equation*}
      		\left(\frac{a^2+b^2}{2}\right)^{1/2}\ge \left(\frac{a^{3/2}+b^{3/2}}{2}\right)^{2/3},
      	\end{equation*}
      	from which we obtain 
      	\begin{equation*}
      		|k_1-k_2|^2\ge (2^{2/3}-1)\left(a^{3/2}+b^{3/2}\right)^{4/3}=(2^{2/3}-1)|k|^2.
      	\end{equation*}
      \end{proof}
      \begin{rem}
      	From the proof, we see that estimate (\ref{v}) holds even without the resonant manifold constraint (\ref{b}). Thus, it can be applied under different restrictions.
      \end{rem}
       
      \begin{lemma}\label{c_22}
      	Define the zero set
      	\[
      	Z_t
      	=
      	\left\{
      	s\in\R:
      	\Delta\omega(k,k_2)\big|_{k_2=se_j+t}=0
      	\right\},
      	\]
      \end{lemma}
      where $\Delta\omega(k,k_2)=|k|^{3/2}-|k-k_2|^{3/2}-|k_2|^{3/2}$. Then we have
      \begin{equation*}
      	|Z_t|=O_d(1).
      \end{equation*}

       \begin{proof}
       	Let
       	\[
       	F(s):=|k|^{3/2}-|k-se_j-t|^{3/2}-|se_j+t|^{3/2},
       	\]
       	so that
       	\[
       	Z_t=\{s\in\mathbb{R}:F(s)=0\}.
       	\]
       	Since
       	\[
       	D^2|x|^{3/2}
       	=
       	\frac{3}{2}|x|^{-1/2}
       	\left(
       	I-\frac12\frac{x\otimes x}{|x|^2}
       	\right)
       	\]
       	is positive definite for every \(x\neq0\), we have
       	\[
       	F''(s)
       	=
       	-e_j^TD^2|k-k_2|^{3/2}e_j
       	-
       	e_j^TD^2|k_2|^{3/2}e_j
       	<0,
       	\]
       	whenever \(k_2\neq0,k\). Thus \(F\) is strictly concave on each connected component of
       	\[
       	\mathbb{R}\setminus\{s:k_2=0\text{ or }k_2=k\}.
       	\]
       	Since a strictly concave function has at most two zeros on each interval, and the line
       	\(\{se_j+t:s\in\mathbb{R}\}\)
       	meets the singular sets \(\{0\}\) and \(\{k\}\) at most once each, it follows that
       	\begin{equation*}
       		|Z_t|=O_d(1),
       	\end{equation*}
       	which is also uniform in $t,k$.
       \end{proof}

		\vspace{15pt}
		\section{Justification of the propagator}
		
		In Remark~\ref{approx}, we explained that the propagator associated with
		\[
		Q_g(t)+\mathcal{R}_g(t)
		\]
		should be understood as the strong limit of a family of regularized propagators. Since both \(Q_{g,b}\) and \(\mathcal{R}_g\) are bounded operators, it suffices to consider the dissipative part \(Q_{g,D}\) and rigorously construct its propagator \(U_{g,D}(t,s)\). We now carry out this construction.
		
		The proof consists of the following four steps:
		\begin{itemize}
			\item Construct a family of regularized operators
			\[
			\{Q_{g,D,\varepsilon}(t)\}_{\varepsilon>0}.
			\]
			
			\item Prove that each \(Q_{g,D,\varepsilon}(t)\) is bounded on \(L^p\), and hence generates a unique propagator \(U_{g,D,\varepsilon}(t,s)\).
			
			\item Establish a uniform bound for the operator norms
			\[
			\|Q_{g,D,\varepsilon}(t)\|_{L^p_1\to L^p},
			\]
			independent of \(\varepsilon\).
			
			\item Show that the family \(\{U_{g,D,\varepsilon}(t,s)\}_{\varepsilon>0}\) converges strongly, and define its limit to be the propagator \(U_{g,D}(t,s)\).
		\end{itemize}
		
		\vspace{7pt}
		We consider the following regularized operators acting on $h$:
		\begin{equation*}
			Q_{g,D, \varepsilon}(t)h:=\iint_{\R^{2d}} e^{-\varepsilon(|k_1|+|k_2|)}4\pi
			|V_{k k_1 k_2}|^2
			\delta(k-k_1-k_2)
			\delta(\omega-\omega_{1}-\omega_{2}) \left[2g_1h_2 \chi_{2\ge 1}-g_1h-g_2h\right]dk_1dk_2
		\end{equation*}
		\begin{equation*}
			-\iint_{\R^{2d}} e^{-\varepsilon(|k|+|k_2|)}4\pi
			|V_{k_1 k k_2}|^2
			\delta(k_1-k-k_2)
			\delta(\omega_{1}-\omega-\omega_{2}) \left[g_2h-g_2h_1\right]dk_1dk_2
		\end{equation*}
		
		\begin{equation*}
			-\iint_{\R^{2d}} e^{-\varepsilon(|k|+|k_1|)}4\pi
			|V_{k_2 k k_1}|^2
			\delta(k_2-k-k_1)
			\delta(\omega_{2}-\omega-\omega_{1}) \left[g_1h-g_1h_2\right]dk_1dk_2.
		\end{equation*}
		\begin{equation*}
			=: Q_{g,D, \varepsilon}^1(t)h+Q_{g,D, \varepsilon}^2(t)+Q_{g,D, \varepsilon}^3(t).
		\end{equation*}

		\vspace{7pt}
		We first show that $Q_{g,D, \varepsilon}(t):L^p(\R^d)\to L^{p}(\R^d)$ is bounded.
		\begin{lemma}\label{new bounded}
		   For $1\le p\le \infty$, $Q_{g,D, \varepsilon}$ is bounded on $L^p(\R^d)$ and satisfies
		   \begin{equation*}
		   	\|Q_{g,D,\varepsilon}\|_{p\to p}\lesssim_{\varepsilon} \|g\|_{\infty,2d+10}.
		   \end{equation*}
		\end{lemma}
		\begin{proof}
			The proof is relatively trivial, since the factors $e^{-\varepsilon(|k_i|+|k_j|)}$ can absorb any polynomial-type factors. Then we can apply Proposition \ref{3.14} to obtain the boundedness.
			
		\end{proof}
		\begin{rem}\label{new dissi}
		From the above boundedness, we know that the regularized operators $Q_{g,D, \varepsilon}$ can generate a unique semigroup $U_{g,D,\varepsilon}(t,s)$. Also note that, following the same proof as in Lemma \ref{dissipative lemma}, the regularized operators $Q_{g,D, \varepsilon}$ are actually dissipative, which implies that $U_{g,D,\varepsilon}(t,s)$ is a contraction, i.e.,
			\begin{equation}\label{contraction}
				\|U_{g,D,\varepsilon}(t,s)h\|_p\le \|h\|_p, \quad \forall h\in L^p(\R^d).
			\end{equation}		
		\end{rem}
		
		\vspace{7pt}
		\begin{lemma}\label{uniform operator}
			For $1\le p\le \infty$, $Q_{g,D,\varepsilon}(t): L_1^{p}(\R^d)\to L^p(\R^d)$ is bounded and satisfies
			\begin{equation*}
			    \|Q_{g,D,\varepsilon}\|_{L_1^p\to L^p}\lesssim \|g\|_{\infty,2d+10},
			\end{equation*}
			where the implicit constant is independent of $\varepsilon$.
		\end{lemma}
		\begin{proof}
			By interpolation, we only need to check the cases $p=1,\infty$. It also suffices to show
			\begin{equation*}
				\|Q_{g,D,\varepsilon}^i\|_{L_1^p\to L^p}\lesssim \|g\|_{\infty,2d+10}, \quad \forall i=1,2,3.
			\end{equation*}
			By definition, the above boundedness can be reduced to 
			\begin{equation*}
				\|\widetilde{Q}_{g,D,\varepsilon}^i\|_{L^p\to L^p}\lesssim \|g\|_{\infty,2d+10}, \quad \forall i=1,2,3,
			\end{equation*}
			where $\widetilde{Q}_{g,D,\varepsilon}^i$ are given by 
			\begin{equation*}
				\widetilde{Q}_{g,D,\varepsilon}^1h:=\iint_{\R^{2d}} e^{-\varepsilon(|k_1|+|k_2|)}4\pi
				|V_{k k_1 k_2}|^2
				\delta(k-k_1-k_2)
				\delta(\omega-\omega_{1}-\omega_{2}) 
			\end{equation*}
			\begin{equation*}
				\times \left[2g_1 \frac{h_2}{\langle 
					k_2\rangle} \chi_{2\ge 1}-g_1\frac{h}{\langle 
					k\rangle}-g_2\frac{h}{\langle 
					k\rangle}\right]dk_1dk_2,
			\end{equation*}
			\begin{equation*}
				\widetilde{Q}_{g,D,\varepsilon}^2h:=-\iint_{\R^{2d}} e^{-\varepsilon(|k|+|k_2|)}4\pi
				|V_{k_1 k k_2}|^2
				\delta(k_1-k-k_2)
				\delta(\omega_{1}-\omega-\omega_{2}) 
			\end{equation*}
			\begin{equation*}
				\times \left[g_2\frac{h}{\langle 
					k\rangle}-g_2\frac{h_1}{\langle 
					k_1\rangle}\right]dk_1dk_2,
			\end{equation*}
			
			\begin{equation*}
				\widetilde{Q}_{g,D,\varepsilon}^3h:=-\iint_{\R^{2d}} e^{-\varepsilon(|k|+|k_1|)}4\pi
				|V_{k_2 k k_1}|^2
				\delta(k_2-k-k_1)
				\delta(\omega_{2}-\omega-\omega_{1}) 
			\end{equation*}
			\begin{equation*}
				\times \left[g_1\frac{h}{\langle 
					k\rangle}-g_1\frac{h_2}{\langle 
					k_2\rangle}\right]dk_1dk_2.
			\end{equation*}
		For $\widetilde{Q}_{g,D,\varepsilon}^1h$, the extra $\langle 
		k\rangle^{-1}$ and $\langle 
		k_2\rangle^{-1}$ can cancel out $|k|$ from the collision kernel $|V_{kk_1k_2}|^2$. We can then follow the same procedure as in Lemma \ref{boundedness of R} to derive
		\begin{equation*}
			\|\widetilde{Q}_{g,D,\varepsilon}^1\|_{L^p\to L^p}\lesssim \|g\|_{\infty,2d+10}.
		\end{equation*}
		For $\widetilde{Q}_{g,D,\varepsilon}^2h$, the second term $-g_2\frac{h_1}{\langle k_1\rangle}$ can be controlled similarly. We decompose the first term as follows:
		\begin{equation*}
			g_2\frac{h}{\langle k\rangle}=g_2\frac{h}{\langle k\rangle}\chi_{0\ge 2}+g_2\frac{h}{\langle k\rangle}\chi_{2>0}.
		\end{equation*}
		Under the restriction $\chi_{0\ge 2}$, we know $|k|\sim |k_1|$, and $\langle k\rangle^{-1}$ can cancel out $|k_1|$ in $|V_{k_1kk_2}|^2$. Under the restriction $\chi_{2>0}$, one has $|k_2|\sim |k_1|$, which enables $g_2$ to absorb any factors.
			
			Analogously, we can obtain the $L^p$-boundedness of $\widetilde{Q}_{g,D,\varepsilon}^3$.
			
		\end{proof}
	
Finally, we verify the strong convergence of $\{U_{g,D,\varepsilon}(t,s)\}_{\varepsilon>0}$. We refer readers to \cite[Lemma A.2]{pan2026local} for a similar discussion.

\begin{lemma}\label{final}
	For $1\le p\le \infty$, the strong limit
	\begin{equation*}
		\lim_{\varepsilon\to 0}U_{g,D,\varepsilon}(t,s)
	\end{equation*}
	exists in $L^p(\R^d)$, for all $0\le s\le t\le T$. Equivalently, for any $h\in L^{p}(\R^d)$, $\{U_{g,D,\varepsilon}(t,s)h\}_{\varepsilon>0}$ is a Cauchy sequence in $L^p(\R^d)$.
\end{lemma}
	\begin{proof}
		Without loss of generality, we can assume $s=0$. Fix an arbitrary small $\eta>0$ and $M=M(\eta)\ge 1$, so that 
		\begin{equation}\label{gh}
			\|\chi(|k|\ge M)h\|_p<\eta.
		\end{equation}
		For the truncated solution $h_{M,\varepsilon}(t):= U_{g,D,\varepsilon}(t,0)\left[\chi(|k|< M)h\right]$, we have the following evolutional equation:
		\begin{equation*}
			\partial_t\left(\langle k\rangle^ah_{M,\varepsilon}(t)\right)=\left(Q_{g,D,\varepsilon}(t)+\mathcal{R}_{g,\varepsilon}(t)\right)\left(\langle k\rangle^ah_{M,\varepsilon}(t)\right),
		\end{equation*}
		where the extra operator $\mathcal{R}_{g,\varepsilon}(t)$ is defined as follows:
		\begin{equation*}
			\mathcal{R}_{g,\varepsilon}(t) \varphi:= \iint_{\R^{2d}} e^{-\varepsilon(|k_1|+|k_2|)}8\pi
			|V_{k k_1 k_2}|^2
			\delta(k-k_1-k_2)
			\delta(\omega-\omega_{1}-\omega_{2}) g_1 \dfrac{\langle k\rangle^{a}-\langle k_2\rangle^{a}}{\langle k_2\rangle^{a}} \varphi_2 \chi_{2\ge 1} \; dk_1dk_2
		\end{equation*}
		\begin{equation*}
			+\iint_{\R^{2d}} e^{-\varepsilon(|k|+|k_2|)}4\pi
			|V_{k_1 k k_2}|^2
			\delta(k_1-k-k_2)
			\delta(\omega_1-\omega-\omega_{2}) g_2 \dfrac{\langle k\rangle^{a}-\langle k_1\rangle^{a}}{\langle k_1\rangle^{a}} \varphi_1 \; dk_1dk_2
		\end{equation*}
		\begin{equation*}
			+\iint_{\R^{2d}} e^{-\varepsilon(|k|+|k_1|)}4\pi
			|V_{k_2 k k_1}|^2
			\delta(k_2-k-k_1)
			\delta(\omega_2-\omega-\omega_{1}) g_1 \dfrac{\langle k\rangle^{a}-\langle k_2\rangle^{a}}{\langle k_2\rangle^{a}} \varphi_2 \; dk_1dk_2.
		\end{equation*}
		
Following the proof of Lemma \ref{boundedness of R}, one can check the $L^p$-boundedness
\begin{equation*}
	\|\mathcal{R}_{g,\varepsilon}(t)\|_{p\to p}\lesssim_a \|g\|_{\infty, 2d+10}.
\end{equation*}
Recall that $Q_{g,D,\varepsilon}$ is dissipative; then we obtain the energy estimate:
\begin{equation}\label{cv}
	\|h_{M,\varepsilon}(t)\|_{p,a}\lesssim M^{a}\|h\|_p\exp\left(C_a t \sup_{t\in [0,T]}\|g(t)\|_{\infty,2d+10}\right), \quad \forall t\in [0,T].
\end{equation}

For any $\varepsilon_1,\varepsilon_2>0$, one has the equation:
\begin{equation*}
	\partial_t\left(h_{M,\varepsilon_1}-h_{M,\varepsilon_2}\right)=Q_{g,D,\varepsilon_1}(t)\left(h_{M,\varepsilon_1}-h_{M,\varepsilon_2}\right)+\left(Q_{g,D,\varepsilon_1}-Q_{g,D,\varepsilon_2}\right)h_{M,\varepsilon_2}.
\end{equation*}
Applying Duhamel's formula and the contraction property (\ref{contraction}), we can derive
\begin{equation*}
	\|h_{M,\varepsilon_1}(t)-h_{M,\varepsilon_2}(t)\|_p\le \int_{0}^t \left\|U_{g,D,\varepsilon_1}(t,s)\left[Q_{g,D,\varepsilon_1}(s)-Q_{g,D,\varepsilon_2}(s)\right]h_{M,\varepsilon_2}\right\|_p ds
\end{equation*}
\begin{equation*}
	\le\int_{0}^t \left\|\left[Q_{g,D,\varepsilon_1}(s)-Q_{g,D,\varepsilon_2}(s)\right]h_{M,\varepsilon_2}\right\|_p ds.
\end{equation*}
Note that the difference $Q_{g,D,\varepsilon_1}-Q_{g,D,\varepsilon_2}$ yields the factor
\begin{equation*}
	\big|e^{-\varepsilon_1(|k_i|+|k_j|)}-e^{-\varepsilon_2(|k_i|+|k_j|)}\big|\lesssim |\varepsilon_1-\varepsilon_2|\left(|k|+|k_1|+|k_2|\right).
\end{equation*}
Then, utilizing the same argument as in Lemma \ref{uniform operator} together with the energy estimate (\ref{cv}) with $a=3$, we can further derive
\begin{equation*}
	\|h_{M,\varepsilon_1}(t)-h_{M,\varepsilon_2}(t)\|_p\lesssim |\varepsilon_1-\varepsilon_2|T\left(\sup_{t\in [0,T]}\|g(t)\|_{\infty,2d+10}\right)\|h_{M,\varepsilon_2}\|_{p,3}
\end{equation*}
\begin{equation*}
	\lesssim  |\varepsilon_1-\varepsilon_2|M^{3}T\|h\|_p\left(\sup_{t\in [0,T]}\|g(t)\|_{\infty,2d+10}\right)\exp\left(C_3 T \sup_{t\in [0,T]}\|g(t)\|_{\infty,2d+10}\right).
\end{equation*}

Thus, by the triangle inequality, we obtain
\begin{equation*}
	\|U_{g,D,\varepsilon_1}(t,0)h-U_{g,D,\varepsilon_2}(t,0)h\|_p\le \|h_{M,\varepsilon_1}(t)-h_{M,\varepsilon_2}(t)\|_p +\|U_{g,D,\varepsilon_1}\left(\chi(|k|\ge M)h\right)\|_p
\end{equation*}
\begin{equation*}
	+\|U_{g,D,\varepsilon_2}\left(\chi(|k|\ge M)h\right)\|_p.
\end{equation*}
Combining the $L^p$-contraction (\ref{contraction}) and (\ref{gh}), we can conclude that $\{U_{g,D,\varepsilon}h\}_{\varepsilon>0}$ satisfies the Cauchy criterion.
		
	\end{proof}

		\newpage
		\section*{Acknowledgement}
		The author is grateful to Jalal Shatah for introducing the capillary wave kinetic equation and recommending several relevant references. The author also thanks Yulin Pan and Xiaoxu Wu for carefully reading the manuscript and for their detailed and valuable comments and suggestions, which have significantly improved the presentation of this paper.
		
		\section*{Conflict of interest statement}
		The author does not have any possible conflict of interest.
		
		\section*{Data availability statement}
		The manuscript has no associated data.
		\bigskip
		\bigskip

		\bibliographystyle{alpha}
		\bibliography{capillary}

\newcommand{\etalchar}[1]{$^{#1}$}
\begin{thebibliography}{KMN14}

\bibitem[AGT16]{alonso2016cauchy}
Ricardo Alonso, Irene~M Gamba, and Minh-Binh Tran.
\newblock The cauchy problem for the quantum boltzmann equation for bosons at
  very low temperature.
\newblock {\em arXiv preprint arXiv:1609.07467}, 2016.

\bibitem[DH21]{deng2021derivation}
Yu~Deng and Zaher Hani.
\newblock On the derivation of the wave kinetic equation for nls.
\newblock In {\em Forum of Mathematics, Pi}, volume~9, page~e6. Cambridge
  University Press, 2021.

\bibitem[DH23a]{deng2023full}
Yu~Deng and Zaher Hani.
\newblock Full derivation of the wave kinetic equation.
\newblock {\em Inventiones mathematicae}, 233(2):543--724, 2023.

\bibitem[DH23b]{deng2023long}
Yu~Deng and Zaher Hani.
\newblock Long time justification of wave turbulence theory.
\newblock {\em arXiv preprint arXiv:2311.10082}, 2023.

\bibitem[ET15]{escobedo2015convergence}
Miguel Escobedo and Minh-Binh Tran.
\newblock Convergence to equilibrium of a linearized quantum boltzmann equation
  for bosons at very low temperature.
\newblock {\em Kinetic and Related Models}, 8(3):493--531, 2015.

\bibitem[Eva22]{evans2022partial}
Lawrence~C Evans.
\newblock {\em Partial differential equations}, volume~19.
\newblock American mathematical society, 2022.

\bibitem[Has62]{hasselmann1962non}
Klaus Hasselmann.
\newblock On the non-linear energy transfer in a gravity-wave spectrum part 1.
  general theory.
\newblock {\em Journal of Fluid Mechanics}, 12(4):481--500, 1962.

\bibitem[KMN14]{kolmakov2014wave}
German~V Kolmakov, Peter Vaughan~Elsmere McClintock, and Sergey~V Nazarenko.
\newblock Wave turbulence in quantum fluids.
\newblock {\em Proceedings of the National Academy of Sciences},
  111(supplement\_1):4727--4734, 2014.

\bibitem[Naz11]{nazarenko2011wave}
Sergey Nazarenko.
\newblock {\em Wave turbulence}, volume 825.
\newblock Springer Science \& Business Media, 2011.

\bibitem[NT18]{nguyen2018kinetic}
Toan~T Nguyen and Minh-Binh Tran.
\newblock On the kinetic equation in zakharov's wave turbulence theory for
  capillary waves.
\newblock {\em SIAM Journal on Mathematical Analysis}, 50(2):2020--2047, 2018.

\bibitem[NT19]{nguyen2019uniform}
Toan~T Nguyen and Minh-Binh Tran.
\newblock Uniform in time lower bound for solutions to a quantum boltzmann
  equation of bosons.
\newblock {\em Archive for Rational Mechanics and Analysis}, 231(1):63--89,
  2019.

\bibitem[P{\etalchar{+}}17]{pan2017understanding}
Yulin Pan et~al.
\newblock {\em Understanding of weak turbulence of capillary waves}.
\newblock PhD thesis, Massachusetts Institute of Technology, 2017.

\bibitem[Paz12]{pazy2012semigroups}
Amnon Pazy.
\newblock {\em Semigroups of linear operators and applications to partial
  differential equations}.
\newblock Springer Science \& Business Media, 2012.

\bibitem[Pei29]{peierls1929kinetischen}
R~Peierls.
\newblock Zur kinetischen theorie der w{\"a}rmeleitung in kristallen.
\newblock {\em Annalen der Physik}, 395(8):1055--1101, 1929.

\bibitem[Pei55]{peierls1955quantum}
Rudolf~Ernst Peierls.
\newblock {\em Quantum theory of solids}.
\newblock Oxford university press, 1955.

\bibitem[PW26]{pan2026local}
Yulin Pan and Xiaoxu Wu.
\newblock Local-in-time existence of $l^{1}$ solutions to the gravity water
  wave kinetic equation.
\newblock {\em arXiv preprint arXiv:2603.10882}, 2026.

\bibitem[PY14]{pan2014direct}
Yulin Pan and Dick~KP Yue.
\newblock Direct numerical investigation of turbulence of capillary waves.
\newblock {\em Physical Review Letters}, 113(9):094501, 2014.

\bibitem[PZ00]{pushkarev2000turbulence}
AN~Pushkarev and VE~Zakharov.
\newblock Turbulence of capillary waves¡ªtheory and numerical simulation.
\newblock {\em Physica D: Nonlinear Phenomena}, 135(1-2):98--116, 2000.

\bibitem[Z{\etalchar{+}}72]{zakharov1972collapse}
Vladimir~E Zakharov et~al.
\newblock Collapse of langmuir waves.
\newblock {\em Sov. Phys. JETP}, 35(5):908--914, 1972.

\bibitem[Zak65]{zakharov1965weak}
Vladimir~E Zakharov.
\newblock Weak turbulence in media with a decay spectrum.
\newblock {\em Journal of Applied Mechanics and Technical Physics},
  6(4):22--24, 1965.

\bibitem[Zak99]{zakharov1999statistical}
V~Zakharov.
\newblock Statistical theory of gravity and capillary waves on the surface of a
  finite-depth fluid.
\newblock {\em European journal of mechanics. B, Fluids}, 18(3):327--344, 1999.

\bibitem[ZF67]{zakharov1967weak}
Vladimir~Evgen'evich Zakharov and NN~Filonenko.
\newblock Weak turbulence of capillary waves.
\newblock {\em Journal of applied mechanics and technical physics},
  8(5):37--40, 1967.

\bibitem[ZLF12]{zakharov2012kolmogorov}
Vladimir~E Zakharov, Victor~S L'vov, and Gregory Falkovich.
\newblock {\em Kolmogorov spectra of turbulence I: Wave turbulence}.
\newblock Springer Science \& Business Media, 2012.

\end{thebibliography}

	\end{document}